\documentclass[12pt,leqno]{article}  % 建檔日期 2015年3月18日
\usepackage{amsmath,amssymb,bbm}
\usepackage{amsthm,array,mathdots}
\usepackage{titlesec} % 重設章節標題所需的巨集
\usepackage[all,cmtip]{xy}
\usepackage{blkarray,arydshln} % 矩陣分隔線所需的巨集

\titlespacing{\section}{0cm}{3.5pc}{1.5pc}

\makeatletter
\def\@citex[#1]#2{\if@filesw\immediate\write\@auxout{\string\citation{#2}}\fi
  \def\@citea{}\@cite{\@for\@citeb:=#2\do
    {\@citea\def\@citea{\@citesep}\@ifundefined
       {b@\@citeb}{{\bf ?}\@warning
       {Citation `\@citeb' on page \thepage \space undefined}}%
{\csname b@\@citeb\endcsname}}}{#1}}
\def\@citesep{; }
\makeatother

\newtheoremstyle{Kang}{}{}{\itshape}{}{\bf}{}{.5em}{}
\theoremstyle{Kang}
\newtheorem{theorem}{Theorem}[section]

\newtheorem{lemma}[theorem]{Lemma}

\newtheorem{prop}[theorem]{Proposition}

\newtheoremstyle{Kremark}{}{}{}{}{\bf}{}{.5em}{}
\theoremstyle{Kremark}
\newtheorem*{remark}{Remark.}
\newtheorem{defn}[theorem]{Definition}

\newtheorem{other}{}

\allowdisplaybreaks[1]  % 允許多行數學式分頁

\def\bm#1{\mathbbm{#1}}

\def\roof#1{\mathop{\ulcorner\mkern-1.5mu#1\mkern-1.5mu\urcorner}}
\DeclareMathOperator{\br}{Br}
\DeclareMathOperator{\dec}{dec}

\title{Degree Three Unramified Cohomology Groups}
\author{Akinari Hoshi$^{(1)}$, Ming-chang Kang$^{(2)}$ and Aiichi Yamasaki$^{(3)}$ \\[3mm]
\begin{minipage}{16cm} \begin{description} \itemsep=-1pt
\item[] $^{(1)}$Department of Mathematics, Niigata University,
Niigata, Japan,\\ E-mail: hoshi@math.sc.niigata-u.ac.jp \item[]
$^{(2)}$Department of Mathematics, National Taiwan University,
Taipei,\\ Taiwan, E-mail: kang@math.ntu.edu.tw \item[]
$^{(3)}$Department of Mathematics, Kyoto University, Kyoto,
Japan,\\ E-mail: aiichi.yamasaki@gmail.com
\end{description} \end{minipage}}
\date{}

\begin{document}

\maketitle

\footnote{\textit{\!\!\! $2010$ Mathematics Subject
Classification}. 14E08, 12G05.} \footnote{\textit{\!\!\! Keywords
and phrases}. Rationality problem, Noether's problem, unramified
cohomology group, unramified Brauer groups, $p$-groups.}
\footnote{\!\!\! This work was partially supported by JSPS KAKENHI
Grant Numbers 24540019, 25400027.}

\footnote{\!\!\! Parts of the work were finished when the
first-named author and the third-named author were visiting the
National Center for Theoretic Sciences (Taipei Office), whose
support is gratefully acknowledged.}

\begin{abstract}
{\noindent\bf Abstract.} Let $k$ be any field, $G$ be a finite
group. Let $G$ act on the rational function field $k(x_g:g\in G)$
by $k$-automorphisms defined by $h\cdot x_g=x_{hg}$ for any
$g,h\in G$. Denote by $k(G)=k(x_g:g\in G)^G$, the fixed subfield.
Noether's problem asks whether $k(G)$ is rational (= purely
transcendental) over $k$. The unramified Brauer group $\br_{\rm
nr}(\bm{C}(G))$ and the unramified cohomology $H_{\rm nr}^3
(\bm{C}(G),\bm{Q}/\bm{Z})$ are obstructions to the rationality of
$\bm{C}(G)$ (see \cite{Sa2} and \cite{CTO}). Peyre proves that, if
$p$ is an odd prime number, then there is a group $G$ such that
$|G|=p^{12}$, $\br_{\rm nr}(\bm{C}(G))=\{0 \}$, but $H_{\rm
nr}^3(\bm{C}(G),\bm{Q}/\bm{Z})\ne \{0 \}$; thus $\bm{C}(G)$ is not
stably $\bm{C}$-rational \cite{Pe2}. Using Peyre's method, we are
able to find groups $G$ with $|G|=p^9$ where $p$ is an odd prime
number such that $\br_{\rm nr}(\bm{C}(G))=\{0 \}$, $H_{\rm
nr}^3(\bm{C}(G),\bm{Q}/\bm{Z})\ne \{0 \}$.
\end{abstract}

\newpage
%------------------------------------S1
\section{Introduction}

Let $k$ be a field, and $L$ be a finitely generated field extension of $k$.
$L$ is called $k$-rational (or rational over $k$) if $L$ is purely transcendental over $k$,
i.e.\ $L$ is isomorphic to some rational function field over $k$.
$L$ is called stably $k$-rational if $L(y_1,\ldots,y_m)$ is $k$-rational for some $y_1,\ldots, y_m$ which are algebraically independent over $L$.
$L$ is called $k$-unirational if $L$ is $k$-isomorphic to a subfield of some $k$-rational field extension of $k$.
It is easy to see that ``$k$-rational" $\Rightarrow$ ``stably $k$-rational" $\Rightarrow$ ``$k$-unirational".

A classical question, the L\"uroth problem by some people, asks whether a $k$-unirational field $L$ is necessarily $k$-rational.
For a survey of the question, see \cite{MT} and \cite{CTS}.

Noether's problem is a special case of the above L\"uroth problem.
Let $k$ be a field and $G$ be a finite group. Let $G$ act on the
rational function field $k(x_g:g\in G)$ by $k$-automorphisms
defined by $h\cdot x_g=x_{hg}$ for any $g,h\in G$. Denote by
$k(G)$ the fixed subfield, i.e.\ $k(G)=k(x_g:g\in G)^G$. Noether's
problem asks, under what situation, the field $k(G)$ is
$k$-rational.

Noether's problem is related to the inverse Galois problem, to the
existence of generic $G$-Galois extensions over $k$, and to the
existence of versal $G$-torsors over $k$-rational field extensions
\cite{Sw}, \cite{Sa1}, \cite[Section 33.1, page 86]{GMS}.

The first counter-example to Noether's problem was constructed by
Swan: $\bm{Q}(C_p)$ is not $\bm{Q}$-rational if $p=47$, 113 or 233
etc.\ where $C_p$ is the cyclic group of order $p$. Noether's
problem for finite abelian groups was studied extensively by Swan,
Voskresenskii, Endo and Miyata, Lenstra, etc. For details, see
Swan's survey paper \cite{Sw}.

In \cite{Sa2}, Saltman defines $\br_{{\rm nr},k}(k(G))$, the
unramified Brauer group of $k(G)$ over $k$. It is known that, if
$k(G)$ is stably $k$-rational, then the natural map $\br(k)\to
\br_{{\rm nr},k}(k(G))$ is an isomorphism; in particular, if $k$
is algebraically closed, then $\br_{{\rm nr},k}(k(G))\allowbreak
=\{0\}$.

In this article, we concentrate on field extensions $L$ over
$\bm{C}$. Thus we will write $\br_{\rm nr}(\bm{C}(G))$ for
$\br_{{\rm nr},\bm{C}}(\bm{C}(G))$, because there is no ambiguity
of the ground field $\bm{C}$. As mentioned before, if $\br_{\rm
nr}(\bm{C}(G))\ne \{0\}$, then $\bm{C}(G)$ is not stably rational
over $\bm{C}$.

%----------------t1.1
\begin{theorem}[{Saltman \cite{Sa2}}] \label{t1.1}
Let $p$ be any prime number. Then there is a group $G$ of order
$p^9$ such that $\br_{\rm nr}(\bm{C}(G))\ne \{0\}$. Consequently
$\bm{C}(G)$ is not stably $\bm{C}$-rational.
\end{theorem}

A convenient formula for computing $\br_{nr}(\bm{C}(G))$ was found
by Bogomolov (\cite[Theorem 3.1]{Bo}). Using this formula,
Bogomolov was able to reduce the group order from $p^9$ to $p^6$.

%-----------------t1.2
\begin{theorem}[{Bogomolov \cite[Lemma 5.6]{Bo}}] \label{t1.2}
Let $p$ be any prime number. Then there is a group $G$ of order
$p^6$ such that $\br_{\rm nr}(\bm{C}(G))\ne \{0\}$.
\end{theorem}

Colliot-Th\'el\`ene and Ojanguren generalized the notion of the
unramified Brauer group to the unramified cohomology group $H_{\rm
nr}^d(\bm{C}(G),\bm{Q}/\bm{Z})$ where $d\ge 2$ \cite{CTO}; also
see Saltman's treatment \cite{Sa3}. Again, if $\bm{C}(G)$ is
stably $\bm{C}$-rational, then $H_{\rm
nr}^d(\bm{C}(G),\bm{Q}/\bm{Z})$ $=\{0\}$ \cite[Proposition
1.2]{CTO}. Moreover, $H_{\rm nr}^2(\bm{C}(G),\bm{Q}/\bm{Z})\simeq
\br_{\rm nr}(\bm{C}(G))$.

Suppose that $G\to GL(W)$ is a faithful complex representation.
Then $\bm{C}(G)$ and $\bm{C}(W)^G$ are stably isomorphic by the
No-Name Lemma (see \cite[Theorem 4.1]{CK} by considering
$\bm{C}(W)(x_g:g\in G)^G$). Thus $H_{\rm
nr}^d(\bm{C}(G),\bm{Q}/\bm{Z})\simeq H_{\rm
nr}^d(\bm{C}(W)^G,\bm{Q}/\bm{Z})$ for any $d\ge 2$ by
\cite[Proposition 1.2]{CTO}. For this reason, we will consider
only $H_{\rm nr}^d(\bm{C}(G),\bm{Q}/\bm{Z})$ for any $d\ge 2$.

Since $\br_{\rm nr}(\bm{C}(G))$ is just an obstruction to the
rationality of $\bm{C}(G)$, it may happen that, for some group
$G$, $\br_{\rm nr}(\bm{C}(G))=\{0\}$, but $\bm{C}(G)$ is not
rational over $\bm{C}$. This phenomenon is exemplified by the
following theorem of Peyre.

%------------------------t1.3
\begin{theorem}[{Peyre \cite[Theorem 3]{Pe2}}] \label{t1.3}
Let $p$ be any odd prime number. Then there is a group $G$ of
order $p^{12}$ such that $\br_{\rm nr}(\bm{C}(G))=\{0\}$ and
$H_{\rm nr}^3(\bm{C}(G),\bm{Q}/\bm{Z})\ne \{0\}$. Consequently,
$\bm{C}(G)$ is not stably $\bm{C}$-rational.
\end{theorem}

The idea of Peyre's proof is to find a subgroup $K_{\max}^3/K^3$
of $H_{\rm nr}^3(\bm{C}(G),\bm{Q}/\bm{Z})$ and to show that
$K_{\max}^3/K^3\ne \{0\}$ (see \cite[page 210]{Pe2}). Using
Peyre's method, we will prove the following theorem.

%--------------------t1.4
\begin{theorem} \label{t1.4}
Let $p$ be an odd prime number. Then there is a group $G$ of order
$p^9$ such that $\br_{\rm nr}(\bm{C}(G))=\{0\}$ and $H_{\rm
nr}^3(\bm{C}(G),\bm{Q}/\bm{Z})\ne \{0\}$. Thus $\bm{C}(G)$ is not
stably $\bm{C}$-rational.
\end{theorem}

Theorem \ref{t1.4} will be proved in Section 2 (see Theorem
\ref{t2.4}, Theorem \ref{t2.6} and Theorem \ref{t2.7}). We will
explain the idea how these ``counter-examples" are constructed in
Section 3. The computation of $K_{\max}^3/K^3$ for extraspecial
groups will also be carried out in Section 3.

%------------------------------------------S2
\section{Main results}

Throughout this article, $p$ is an odd prime number, $\bm{F}_p$ is the finite field with $p$ elements.

Recall the construction of the $p$-group $G$ in \cite[Section 5]{Pe2}.
$G$ is a $p$-group of exponent $p$ satisfying that (i) the center of $G$ is equal to its commutator subgroup,
and (ii) $G$ is a central extension of vector spaces over $\bm{F}_p$.
Thus there are finite-dimensional vector spaces $V$ and $U$ over $\bm{F}_p$ and a short exact sequence
\[
0\to V \xrightarrow{\iota} G \xrightarrow{\pi} U\to 0
\]
such that $\iota(V)=Z(G)=[G,G]$ (where $Z(G)$ and $[G,G]$ denote the center of $G$ and the commutator subgroup of $G$ respectively).

We will adopt the multiplicative notations for elements of $G$,
and the additive notation for elements in the vector spaces $V$ and $U$,
$V^*$ and $U^*$ will denote the dual spaces of $V$ and $U$.

Define $\gamma:\wedge^2 U\to V$ by
\begin{equation} \label{q1}
\iota\circ \gamma (\pi(g_1)\wedge \pi(g_2))=[g_1,g_2]
\end{equation}
for any $g_1,g_2\in G$ where $[g_1,g_2]:=g_1g_2g_1^{-1}g_2^{-1}$.

Since $\iota(V)=[G,G]$, $\gamma$ is surjective.
It follows that the dual map $\gamma^*:V^*\to \wedge^2 U^*$ is injective (note that we write $\wedge^2 U^*$ for $\wedge^2 (U^*)$).

%-------------------d2.1
\begin{defn} \label{d2.1}
For any positive integer $d$, define $\Phi:\wedge^d U^*\to (\wedge^d U)^*$ as follows.
For $f_1,\ldots,f_d\in U^*$ and $f=f_1\wedge f_2\wedge \cdots \wedge f_d$,
define $\Phi(f)=\varphi_f:\wedge^d U\to \bm{F}_p$ such that $\varphi_f(u_1\wedge u_2\wedge \cdots \wedge u_d)= %
\sum_{\tau\in S_d}\varepsilon(\tau)f_1(u_{\tau(1)})\cdot
f_2(u_{\tau(2)})\cdot \cdots \cdot f_d(u_{\tau(d)})$ for any
$u_1,\ldots,u_d\in U$ (see \cite[page 209]{Pe2}). Thus we define
the non-degenerate pairing $\wedge^d U\times \wedge^d U^* \to
\bm{F}_p$ by $\langle \langle s,f\rangle \rangle:=\Phi(f)(s)$ for
any $s\in \wedge^d U$, any $f\in \wedge^d U^*$.
\end{defn}

%------------------d2.2
\begin{defn}[{\cite[page 209]{Pe2}}] \label{d2.2}
Let $\gamma$ be the map of Formula \eqref{q1}.
Define
\begin{gather*}
K^2=\gamma^*(V^*),\quad K^3=\gamma^*(V^*)\wedge U^*, \\
S^2=(K^2)^\bot=\{w\in \wedge^2 U:\langle\langle w,f\rangle\rangle=0 \mbox{ for all }f\in K^2\}, \\
S^3=(K^3)^\bot=\{w\in \wedge^3 U:\langle\langle
w,f\rangle\rangle=0 \mbox{ for all }f\in K^3\}.
\end{gather*}

We define
\begin{gather*}
S_{\dec}^2=\langle  u_1\wedge u_2\in S^2:u_1,u_2\in U  \rangle, \\
S_{\dec}^3=\langle  u'\wedge u\in S^3:u'\in \wedge^2 U,~u\in U  \rangle, \\
K_{\max}^2=(S_{\dec}^2)^\bot, \quad K_{\max}^3=(S_{\dec}^3)^\bot
\end{gather*}
where $(S_{\dec}^2)^\bot$ is the orthogonal complement of
$S_{\dec}^2$ in the pairing $\wedge^2U\times \wedge^2 U^* \to
\bm{F}_p$, similarly for $(S_{\dec}^3)^\bot$.
\end{defn}

%-------------------t2.3
\begin{theorem}[{Peyre \cite[Theorem 2, page 210]{Pe2}}] \label{t2.3}
Let $G$ be a $p$-group defined as above. Then $K_{\max}^2/K^2
\simeq \br_{\rm nr}(\bm{C}(G))$ and $K_{\max}^3 /K^3$ is a
subgroup of $H_{\rm nr}^3(\bm{C}(G),\bm{Q}/\bm{Z})$.
\end{theorem}

The main results of this paper are the following theorem
\ref{t2.4}, \ref{t2.6} and \ref{t2.7}.

%----------------t2.4
\begin{theorem} \label{t2.4}
Let $p$ be an odd prime number, $G$ be the $p$-group of exponent $p$ defined by $G=\langle v_i, u_j: 1\le i\le 3, 1\le j\le 6\rangle$ satisfying the following conditions

$(1)$ $Z(G)=[G,G]=\langle v_1,v_2,v_3\rangle$, and

$(2)$ $[u_1,u_2]=[u_3,u_4]=v_1$, $[u_1,u_4]=[u_2,u_5]=[u_3,u_6]=v_2$, $[u_3,u_5]=[u_4,u_6]=v_3$,
and the other unlisted commutators, e.g.\ $[u_1,u_3]$, $[u_1,u_5]$, etc.,
are equal to the identity element of $G$.

Then $\br_{\rm nr}(\bm{C}(G))=\{0 \}$ and $H_{\rm
nr}^3(\bm{C}(G),\bm{Q}/\bm{Z})\ne \{0 \}$.
\end{theorem}

\begin{proof}
Because of Theorem \ref{t2.3}, it suffices to show that
$K_{\max}^2/K^2=\{0\}$ and $K_{\max}^3/K^3\ne \{0\}$ (remember
that we write $V=\langle v_1,v_2,v_3\rangle$, $U=\langle u_i:1\le
i\le 6\rangle$ and $0\to V \xrightarrow{\iota} G\xrightarrow{\pi}
U\to 0$).

Since $K_{\max}^2$ and $K^2$ are dual to $S_{\dec}^2$ and $S^2$, it is enough to sow that $S_{\dec}^2=S^2$.
Similarly, it is enough to show that $S_{\dec}^3 \subsetneq S^3$.

\bigskip
Step 1.
Let $\{u_j^*:1\le j\le 6\}$ be the dual basis of $\{u_j:1\le j\le 6\}$, and $\{v_i^*:1\le i\le 3\}$ be the dual basis of $\{v_i:1\le i\le 3\}$.

By the definition of the group $G$, the map $\gamma:\wedge^2 U\to V$ is defined by $\gamma (u_1 \wedge u_2)=\gamma (u_3\wedge u_4)=v_1$,
$\gamma(u_1\wedge u_4)=\gamma(u_2\wedge u_5)=\gamma(u_3\wedge u_6)=v_2$, $\gamma(u_3\wedge u_5)=\gamma(u_4\wedge u_6)=v_3$,
and $\gamma(u_i\wedge u_j)=0$ for the remaining $u_i\wedge u_j$ ($1\le i< j\le 6$).

It is easy to verify that $\gamma^*:V^* \to \wedge^2 U^*$ is given by $\gamma^*(v_1^*)=u_1^*\wedge u_2^*+u_3^*\wedge u_4^*$,
$\gamma^*(v_2^*)=u_1^*\wedge u_4^*+u_2^*\wedge u_5^*+u_3^*\wedge u_6^*$,
$\gamma^*(v_3^*)=u_3^*\wedge u_5^*+u_4^*\wedge u_6^*$.

It follows that $K^2=\gamma^*(V^*)=\langle u_1^*\wedge u_2^*+u_3^*\wedge u_4^*$, $u_1^*\wedge u_4^*+u_2^*\wedge u_5^*+u_3^*\wedge u_6^*$,
$u_3^*\wedge u_5^*+u_4^*\wedge u_6^*\rangle$.

\bigskip
Step 2.
We will show that $S^2=S_{\dec}^2$.

Note that $S^2=(K^2)^\bot$ and $\dim_{\bm{F}_p}
S_2=\dim_{\bm{F}_p} (\wedge^2 U^*)-\dim_{\bm{F}_p} K^2=12$. We
will find a basis of $S^2$.

For the convenience of notation, we will write $(1,2)$ for $u_1
\wedge u_2$; thus $(1,2)-(3,4)$ denotes $u_1\wedge u_2-u_3\wedge
u_4$.

Since the three basis elements of $K^2$ have been found,
it is not difficult to verify the following elements belong to $S^2=(K^2)^\bot$:
\begin{gather} \label{q2}
\begin{split}
& (1,2)-(3,4),~ (1,3),~ (1,4)-(2,5),~ (1,5),~ (1,6),~ (2,3),~ (2,4), \\
& (2,5)-(3,6),~ (2,6),~ (3,5)-(4,6),~ (4,5),~ (5,6).
\end{split}
\end{gather}

They are 12 linearly independent elements in $S^2$.
Hence they are the basis elements of $S^2$.

Among the 12 vectors in Formula \eqref{q2}, except for
$(1,2)-(3,4)$, $(1,4)-(2,5)$, $(2,5)-(3,6)$ and $(3,5)-(4,6)$, the
remaining vectors (e.g.\ $(1,3)$, $(1,5)$, etc.) obviously belong
to $S_{\dec}^2$. We will show that the ``exceptional" four vectors
also belong to $S_{\dec}^2$.

Note that $(u_1+u_4)\wedge (u_2+u_3)=[(1,2)-(3,4)]+(1,3)-(2,4)\in S^2$.
Thus $(u_1+u_4)\wedge (u_2+u_3) \in S_{\dec}^2$.
It follows that $(1,2)-(3,4)\in S_{\dec}^2$.

Similarly, use the formula $(u_1+u_2)\wedge (u_4-u_5)=[(1,4)-(2,5)]-(1,5)+(2,4)$.
We find $(1,4)-(2,5)\in S_{\dec}^2$.
Use the formula $(u_2+u_6)\wedge (u_3+u_5)=[(2,5)-(3,6)]+(2,3)-(5,6)$.
We find that $(2,5)-(3,6)\in S_{\dec}^2$.

Finally, $(u_2+u_3+u_4)\wedge
(u_5-u_6)=[(3,5)-(4,6)]+[(2,5)-(3,6)]-(2,6)+(4,5)$. Since we have
shown that $(2,5)-(3,6)\in S_{\dec}^2$, it follows that
$(3,5)-(4,6)\in S_{\dec}^2$ also. Done.

\bigskip
Step 3. We will show that $S_{\dec}^3 \subsetneq S^3$. In fact, we
will show that $\dim_{\bm{F}_p} S^3/S_{\dec}^3=1$.

By Step 1, $\gamma^*(V^*)=\langle f_1,f_2,f_3\rangle$ where $f_1=u_1^*\wedge u_2^*+u_3^*\wedge u_4^*$,
$f_2=u_1^*\wedge u_4^*+u_2^*\wedge u_5^*+u_3^*\wedge u_6^*$,
$f_3=u_3^*\wedge u_5^*+u_4^*\wedge u_6^*\rangle$.
Thus $K^3=\gamma^*(V^*)\wedge U^*=\langle f_i\wedge u_j^*: 1\le i\le 3, 1\le j\le 6\rangle$.

We will write $[i,j,k]$ for $u_i^*\wedge u_j^*\wedge u_k^*$. Write
down explicitly the eighteen generators $f_i\wedge u_j^*$ ($1\le
i\le 3$, $1\le j\le 6$) as elements in $A\cup B$ where $A$ is the
set consisting of the vectors
\[
[1,2,3],~ [1,2,4],~ [1,3,4],~ [2,3,4],~ [3,4,5],~ [3,4,6],~ [3,5,6],~ [4,5,6],
\]
and $B$ is the set consisting of the vectors
\begin{align*}
& [1,2,5]+[3,4,5],~[1,2,6]+[3,4,6],~[1,3,5]+[1,4,6], \\
& [1,2,5]+[1,3,6],~ -[1,2,4]+[2,3,6],~ -[1,3,4]-[2,3,5], \\
& [2,3,5]+[2,4,6],~ [2,4,5]+[3,4,6],~[1,4,5]-[3,5,6],~ [1,4,6]+[2,5,6].
\end{align*}

Since $[3,4,6]\in A$, the generator $[1,2,6]+[3,4,6]$ in $B$ may
be replaced by $[1,2,6]$. Simplify the vectors of $B$ by this way.
We find that $K^3$ is generated by vectors in the following set
\begin{equation} \label{q3}
(\{ [i,j,k]: 1\le i<j<k \le 6\}\backslash C)\cup D
\end{equation}
where $C=\{[1,3,5],[1,4,6],[1,5,6],[2,5,6]\}$ and $D=\{[1,3,5]+[1,4,6],[1,4,6]+[2,5,6]\}$.

The 18 vectors in Formula \eqref{q3} are linearly independent over $\bm{F}_p$.
Hence $\dim_{\bm{F}_p} K^3=18$ and $\dim_{\bm{F}_p} S^3=\dim_{\bm{F}_p} (\wedge^3 U^*)-\dim_{\bm{F}_p} K^3=2$.

It is clear that $w_1,w_2\in (K^3)^\bot$ where $w_1=u_1\wedge
u_5\wedge u_6$ and $w_2=u_1\wedge u_3\wedge u_5-u_1\wedge
u_4\wedge u_6+u_2\wedge u_5\wedge u_6$. Since $\dim_{\bm{F}_p}
S^3=2$, it follows that $S^3=\langle w_1,w_2\rangle$.

\bigskip
Step 4.
We will show that $S_{\dec}^3=\langle w_1\rangle$, which will finish the proof that $\dim_{\bm{F}_p} S^3/S_{\dec}^3\allowbreak =1$.

Recall that $S_{\dec}^3=\langle u'\wedge u\in S^3:u'\in \wedge^2 U,u\in U\rangle$.
Elements in $S^3$ of the form $u'\wedge u$ where $u'\in \wedge^2 U$ and $u\in U$ will be called eligible elements of $S_{\dec}^3$.
We will show that, up to a scalar multiple, $w_1$ is the only one eligible element of $S_{\dec}^3$.
This will finish the proof that $S_{\dec}^3=\langle w_1\rangle$.

Suppose that $w\in S_{\dec}^3$ is a non-zero eligible element.
Then $w\in S^3=\langle w_1,w_2\rangle$ and $w=u'\wedge u$ for some $u'\in \wedge^2 U$ and $u\in U$.
Write $w=a_1\cdot w_1+a_2\cdot w_2$ for some $a_1,a_2\in \bm{F}_p$.
We will show that $a_2=0$.

Since $w=u'\wedge u$ for some $u'\in \wedge^2 U$ and $u\in U$, apply the following Lemma \ref{l2.5}.
It is necessary that $w\wedge u_0=0$ for some non-zero vector $u_0\in U$.
Write $u_0=\sum_{1\le j\le 6} b_j \cdot u_j$ where $b_j\in \bm{F}_p$.
Expand the relation $(\sum_{1\le i\le 2} a_i\cdot w_i)\wedge (\sum_{1\le j\le 6} b_j\cdot u_j)=0$.
A non-trivial solution for $(a_1,a_2,b_1,\ldots,b_6)$ is of the following form
\[
(0,0,b_1,b_2,\ldots,b_6),~(a_1,a_2,0,0,0,0,0,0),~(a_1,0,b_1,0,0,0,b_5,b_6).
\]
If we require that $a_1w_1+a_2w_2\ne 0$ and $\sum_{1\le j\le 6} b_j\cdot u_j\ne 0$,
it is necessary that $a_2=0$ as we expected before.
\end{proof}

%-----------------------l2.5
\begin{lemma}[{\cite[page 265]{Pe1}}] \label{l2.5}
Let $d$ be a positive integer and $U$ be a vector space over
$\bm{F}_p$ such that $d\le \dim_{\bm{F}_p}U$. Suppose that $w\in
\wedge^d U$ is a non-zero vector. Then $w=u'\wedge u$ for some
$u'\in \wedge^{d-1} U$ and $u\in U$ if and only if there is a
non-zero vector $u_0$ such that $w\wedge u_0=0$.
\end{lemma}

\begin{proof}
$\Rightarrow$ If $w=u'\wedge u$, then $u\ne 0$. Thus $w\wedge u=0$.

$\Leftarrow$ Suppose $u_0$ is a non-zero vector and $w\wedge u_0=0$.
Let $u_1,u_2,\ldots,u_n$ be a basis of $U$ with $u_1=u_0$.
Write $w=\sum_\lambda a_\lambda u_{\lambda_1} \wedge u_{\lambda_2} \wedge \cdots \wedge u_{\lambda_d}$ where $\lambda$ runs over the ordered $d$-subsets
$(\lambda_1,\lambda_2,\ldots,\lambda_d)$ with $\lambda_1<\lambda_2<\cdots<\lambda_d$ and $a_\lambda\in\bm{F}_p$.

If $w\wedge u_1=0$, then $a_\lambda=0$ if $\lambda=(\lambda_1,\ldots,\lambda_d)$ with $\lambda_1\ge 2$.
Hence we may write $w=u'\wedge u_1$ for some $u'\in \wedge^{d-1}U$.
\end{proof}

%------------------t2.6
\begin{theorem} \label{t2.6}
Let $p$ be an odd prime number, $t\in \bm{F}_p \backslash \{0\}$.
Let $G$ be the $p$-group of exponent $p$ defined by $G=\langle v_i,u_j:1\le i\le 3, 1\le j\le 6\rangle$ satisfying the following conditions

$(1)$ $Z(G)=[G,G]=\langle v_1,v_2,v_3\rangle$, and

$(2)$ $[u_1,u_2]=[u_4,u_5]^{-1}=v_1$,
$[u_2,u_3]=[u_5,u_6]^{-1}=[u_1,u_4]=v_2$,
$[u_3,u_6]=[u_2,u_4]^{-1}=v_3$, $[u_1,u_5]=v_3^t$, and the other
unlisted commutators are equal to the identity element of $G$.

Then ${\rm Br}_{\rm nr}(\bm{C}(G))=\{0\}$. Moreover,
$K_{\max}^3/K^3\ne \{0\}$ if and only if $t\in \bm{F}_p \backslash
\bm{F}_p^2$. If $K_{\max}^3/K^3 \ne \{0\}$, then $\dim_{\bm{F}_p}
K_{\max}^3/K^3=2$.
\end{theorem}

\begin{proof}
The proof is similar to that of Theorem \ref{t2.4}.

Step 1.
Let $\gamma:\wedge^2 U\to V$ and $\gamma^*:V^*\to \wedge^2 U^*$ be the maps.
Then $K^2=\gamma^2(V^*)=\langle f_1,f_2,f_3\rangle$ where $f_1=u_1^*\wedge u_2^*-u_4^*\wedge u_5^*$,
$f_2=u_2^*\wedge u_3^*-u_5^*\wedge u_6^*+u_1^*\wedge u_4^*$,
$f_3=u_3^*\wedge u_6^*-tu_1^*\wedge u_5^*-u_2^*\wedge u_4^*$.

We adopt the abbreviation $(i,j)=u_i\wedge u_j$ for $1\le i< j\le 6$.

Then $S^2=(K^2)^\bot$ is generated by the 12 basis elements
\[
(1,3),~ (1,6),~ (2,5),~ (2,6),~ (3,4),~ (3,5),~ (4,6)
\]
and
\[
(1,2)+(4,5),~ (1,4)+(5,6),~ (1,5)+t(3,6),~ (2,3)+(5,6),~ (2,4)+(3,6).
\]

Use the relations
\begin{align*}
(u_1+tu_3)\wedge (u_5+u_6) &= [(1,5)+t(3,6)]+(1,6)+t(3,5), \\
(u_2+u_5)\wedge (u_3+u_6) &= [(2,3)+(5,6)]+(2,6)-(3,5), \\
(u_2+u_3)\wedge (u_4+u_6) &= [(2,4)+(3,6)]+(2,6)+(3,4), \\
(u_1-u_6)\wedge (tu_3+u_4+u_5) &= [(1,4)+(5,6)]+[(1,5)+t(3,6)]+t(1,3)+(4,6), \\
(u_1+u_5)\wedge (u_2-u_4-u_6) &=
[(1,2)+(4,5)]-[(1,4)+(5,6)]-(1,6)-(2,5).
\end{align*}

It follows that $S^2=S_{\dec}^2$ and thus
$\br_{nr}(\bm{C}(G))=\{0\}$.

\bigskip
Step 2.
We will show that $\dim_{\bm{F}_p} S^3=2$, $\dim_{\bm{F}_p} S_{\dec}^3=0$ if $t\in \bm{F}_p \backslash \bm{F}_p^2$,
$\dim_{\bm{F}_p} S_{\dec}^3=2$ if $t\in \bm{F}_p^2$.

We will find a basis of $K^3=\gamma^*(V^*)\wedge U^*$. Use the
abbreviation $[i,j,k]=u_i^*\wedge u_j^*\wedge u_k^*$. Write down
explicitly the 18 generators $f_i\wedge u_j^*$ ($1\le i\le 3$,
$1\le j\le 6$) where $f_1$, $f_2$, $f_3$ are defined in Step 1.
They are the following vectors
\begin{align*}
& [1,4,5],~ [2,4,5],~ [1,2,4],~ [1,2,5], \\
& [1,2,3]-[3,4,5],~ [1,2,6]-[4,5,6],~ [1,2,3]-[1,5,6], \\
& [1,2,4]+[2,5,6],~ [1,3,4]+[3,5,6],~ [2,3,4]-[4,5,6], \\
& [1,4,5]+[2,3,5],~ [1,4,6]+[2,3,6],~ [1,2,4]-[1,3,6], \\
& t[1,2,5]+[2,3,6],~ t[1,3,5]+[2,3,4],~ t[1,4,5]-[3,4,6], \\
& [2,4,5]+[3,5,6],~ t[1,5,6]+[2,4,6].
\end{align*}

Simplify the above vectors as Step 3 in the proof of Theorem
\ref{t2.4}. We get the following vectors
\begin{gather} \label{q4}
\begin{split}
& [1,2,4],~ [1,2,5],~ [1,3,4],~ [1,3,6],~ [1,4,5],~ [1,4,6], \\
& [2,3,5],~ [2,3,6],~ [2,4,5],~ [2,5,6],~ [3,4,6],~ [3,5,6], \\
& [1,2,3]-[3,4,5],~ [1,2,3]-[1,5,6],~ [1,2,6]-[4,5,6], \\
& [2,3,4]-[4,5,6],~ t[1,3,5]+[2,3,4],~ t[1,5,6]+[2,4,6].
\end{split}
\end{gather}

It is not difficult to verify that vectors in Formula \eqref{q4} are linearly independent over $\bm{F}_p$.
Hence $\dim_{\bm{F}_p} K^3=18$.
Thus $\dim_{\bm{F}_p} S^3=2$.

Define $w_1,w_2\in \wedge^3 U$ by
\begin{align*}
w_1 &= u_1\wedge u_2\wedge u_3+u_1\wedge u_5\wedge u_6-tu_2\wedge u_4\wedge u_6+u_3\wedge u_4\wedge u_5, \\
w_2 &= tu_1\wedge u_2\wedge u_6+tu_2\wedge u_3\wedge u_4+tu_4\wedge u_5\wedge u_6-u_1\wedge u_3\wedge u_5.
\end{align*}

Clearly $w_1,w_2\in (K^3)^\bot=S^3$.
It follows that $S^3=\langle w_1,w_2\rangle$.

\bigskip
Step 3. We will calculate $S_{\dec}^3$.

As in Step 4 of the proof of Theorem \ref{t2.4},
we call an element $w\in S_{\dec}^3$ an eligible element if $w\in S^3$ and $w=u'\wedge u$ for some $u'\in \wedge^2 U$ and $u\in U$.

If a non-zero vector $w\in S_{\dec}^3$ is eligible, write $w=a_1w_1+a_2w_2$ where $a_1,a_2\in \bm{F}_p$.
Apply Lemma \ref{l2.5}, there is a non-zero vector $u_0\in U$ such that $w\wedge u_0=0$.
Write $u_0=\sum_{1\le j\le 6} b_j\cdot u_j$ where $b_j \in \bm{F}_p$.
Find the non-trivial solutions $(a_1,a_2,b_1,b_2,b_3,b_4,b_5,b_4)$ satisfying $(\sum_{1\le i\le 2} a_iw_i)\wedge (\sum_{1\le j\le 6} b_ju_j)=0$.

If $t\in \bm{F}_p\backslash \bm{F}_p^2$,
the only non-trivial solutions of $(a_1,a_2,b_1,b_2,b_3,b_4,b_5,b_6)$ are $(a_1,a_2,\allowbreak 0,0,0,0,0,0)$ and $(0,0,b_1,b_2,b_3,b_4,b_5,b_6)$.
Thus no non-zero eligible elements exist at all.
Hence $S_{\dec}^3=\{0\}$.

If $t\in\bm{F}_p^2$, write $t=1/c^2$ where $c\in \bm{F}_p\backslash \{0\}$.
The non-trivial solutions of $(a_1,a_2,b_1,\allowbreak b_2,b_3,b_4,b_5,b_6)$ are $(a_1,a_2,0,0,0,0,0,0)$, $(0,0,b_1,b_2,b_3,b_4,b_5,b_6)$ and
\[
(a,\varepsilon ac,\varepsilon b_1c,b_2,\varepsilon b_3 c, b_1,\varepsilon b_2 c, b_3)
\]
where $a\in\bm{F}_p\backslash \{0\}$, $\varepsilon\in \{1,-1\}$, $(b_1,b_2,b_3)$ is a non-zero vector in $\bm{F}_p^3$.
In conclusion, there are essentially two eligible elements $w_1+cw_2$ and $w_1-cw_2$.
Thus $S_{\dec}^3=\langle w_1+cw_2,\allowbreak w_1-cw_2\rangle=\langle w_1,w_2\rangle=S^3$.
\end{proof}

The degenerate case $t=0$ can be proved similarly. We record it as
the following theorem.

%------------------t2.7
\begin{theorem} \label{t2.7}
Let $p$ be an odd prime number. Let $G$ be the $p$-group of
exponent $p$ defined by $G=\langle v_i,u_j:1\le i\le 3, 1\le j\le
6\rangle$ satisfying the following conditions

$(1)$ $Z(G)=[G,G]=\langle v_1,v_2,v_3\rangle$, and

$(2)$ $[u_1,u_2]=[u_4,u_5]^{-1}=v_1$,
$[u_2,u_3]=[u_5,u_6]^{-1}=[u_1,u_4]=v_2$,
$[u_3,u_6]=[u_2,u_4]^{-1}=v_3$, and the other unlisted commutators
are equal to the identity element of $G$.

Then ${\rm Br}_{\rm nr}(\bm{C}(G))=\{0\}$ and $K_{\max}^3/K^3\ne
\{0\}$. In fact, $\dim_{\bm{F}_p} K_{\max}^3/K^3=1$.
\end{theorem}

\begin{proof}

Note that $\gamma^*(V^*)=\langle f_1,f_2,f_3\rangle$ where
$f_1=u_1^*\wedge u_2^*-u_4^*\wedge u_5^*$, $f_2=u_2^*\wedge
u_3^*-u_5^*\wedge u_6^*+u_1^*\wedge u_4^*$, $f_3=u_3^*\wedge
u_6^*-u_2^*\wedge u_4^*$.

It can be shown that $S^2=S_{\dec}^2$; a heuristic ``proof" is by
setting $t=0$ in Step 1 of the proof of Theorem \ref{t2.6}.

For the proof that $\dim_{\bm{F}_p} S^3/S_{\dec}^3=1$, it is not
difficult to show that $K^3$ is generated by
\begin{align*}
& [1,2,4],~ [1,2,5],~ [1,2,6],~ [1,3,4],~ [1,3,6],~ [1,4,5],~ [1,4,6], \\
& [2,3,4],~ [2,3,5],~ [2,3,6],~ [2,4,5],~ [2,4,6],~ [2,5,6],~ [3,4,6], \\
& [3,5,6],~ [4,5,6],~ [1,2,3]-[3,4,5],~ [1,2,3]-[1,5,6].
\end{align*}

Thus $S^3=\langle w_1,w_2\rangle$ where $w_1=u_1\wedge u_2\wedge
u_3+u_1\wedge u_5\wedge u_6+u_3\wedge u_4\wedge u_5$,
$w_2=u_1\wedge u_3\wedge u_5$. By the same method as above, we
find that $S_{\dec}^3=\langle w_2\rangle$.
\end{proof}

%--------------------------------------------------S3
\section{Further remarks}

For any prime number $p$, an extraspecial $p$-group $G$ is a group $G$ such that $Z(G)=\langle v\rangle \simeq C_p$
and $G/\langle v\rangle$ is an elementary abelian group of order $p^{2n}$ where $n\ge 1$ \cite[pages 203--208]{Go}.
Thus $G$ may be presented as $0\to V\xrightarrow{\iota} G\xrightarrow{\pi} U\to 0$ where $V=\langle v\rangle$
and $U$ is a vector space over $\bm{F}_p$ with $\dim_{\bm{F}_p} U=2n$.
It can be shown that there are basis elements $u_1,u_2,\ldots, u_{2n}$ of $U$ such that,
within $G$, $[u_{2i-1},u_{2i}]=v$ for $1\le i\le n$ and $[u_{j,l}]=1$ if $l-j\ge 1$ and $(j,l)\ne (2i-1,2i)$ for some $i$.
From the above definition, the exponent of $G$ is $p$ or $p^2$.

For any prime number $p$, there are precisely two non-isomorphic
non-abelian groups $G$ with order $p^3$. Both of them are
extraspecial $p$-groups. It is known that $\bm{C}(G)$ is rational
over $\bm{C}$ for such groups $G$ \cite{CK}. Thus $\br_{\rm
nr}(\bm{C}(G))=\{0\}=H_{\rm nr}^d(\bm{C}(G),\bm{Q}/\bm{Z})$ for
all $d\ge 3$ by \cite[Proposition 1.2]{CTO}. A direct computation
for $H_{\rm nr}^3(\bm{C}(G),\bm{Q}/\bm{Z})=\{0\}$ may be found in
Black's paper \cite{Bl}.

If $p$ is an odd prime number and $G$ is an extraspecial $p$-group
of order $p^{2n+1}$, it is known that $\br_{\rm
nr}(\bm{C}(G))=\{0\}$ \cite{KK}. Now we turn to the degree three
unramified cohomology group.

%---------------p3.1
\begin{prop} \label{p3.1}
Let $p$ be an odd prime number, $G$ be the extraspecial $p$-group of exponent $p$ and of order $p^{2n+1}$ where $n\ge 1$.
We present $G$ as $0\to V\xrightarrow{\iota} G\xrightarrow{\pi} U\to 0$ as above where $V=\langle v\rangle$, $U=\langle u_1,\ldots,u_{2n}\rangle$,
$[u_{2i-1},u_{2i}]=v$ for $1\le i\le n$.
Define $\gamma: \wedge^2 U\to V$ by Formula \eqref{q1} in Section 2.
Then $K_{\max}^3/K^3=\{0\}$.
\end{prop}

\begin{proof}
Let $u_1^*,\ldots,u_{2n}^*$ be the dual basis of $u_1,\ldots,u_{2n}$.
Then $\gamma^*(V^*)=\langle f\rangle$ where $f=\sum_{1\le i\le n} u_{2i-1}^* \wedge u_{2i}^*$.

\bigskip
When $n=1$, it is easy to see that $K^3=\{0\}$ and $S^3=\wedge^3
U$. Thus $S_{dec}^3=S^3$. From now on, we assume that $n \ge 2$.

\bigskip
Step 1.
We will determine $K^3=\gamma^*(V^*)\wedge U^*=\langle f\wedge u_j^*:1\le j\le 2n\rangle$.
We will show that $f\wedge u_1^*, f\wedge u_2^*,\ldots,f\wedge u_{2n}^*$ are linearly independent in $\wedge^3 U^*$.
Thus $\dim_{\bm{F}_p} K^3=2n$.

Note that $f\wedge u_j^*=\sum_{1\le i\le n} u_{2i-1}^*\wedge
u_{2i}^*\wedge u_j^*$. On the other hand, $u_{2l-1}^*\wedge
u_{2l}^*\wedge u_j^*=0$ if $j=2l-1$ or $2l$. Hence $f\wedge u_1^*$
has a term with non-vanishing coefficient, e.g.\ $u_3^*\wedge
u_4^*\wedge u_1^*$, which doesn't appear in $f\wedge u_j^*$ if
$j\ge 2$. Similar facts are valid for other $f\wedge u_j^*$. Hence
these $f\wedge u_j^*$ (where $1\le j\le 2n$) are linearly
independent.

\bigskip
Step 2. We will find linearly independent elements in $S^3$ whose
total number is $\dim_{\bm{F}_p} (\wedge^3 U^*)-2n=\dim_{\bm{F}_p}
S^3$. Hence they form a basis of $S^3$.

Note that, if $w:=u_1\wedge u_3\wedge u_5$, $u_1\wedge u_3\wedge
u_6$, $u_1\wedge u_2\wedge u_5$, $u_2\wedge u_3\wedge u_6$, etc.,
then $\langle \langle w,f\wedge u_j^* \rangle \rangle =0$ in the
pairing defined in Definition \ref{d2.1}, because $\langle
 \langle u_1\wedge u_3\wedge u_5,u_{2i-1}^*\wedge u_{2i}^*\wedge
u_j^* \rangle \rangle =0$ where $u_{2i-1}^*\wedge u_{2i}^*\wedge
u_j^*$ is a standard term of $f\wedge u_j^*$. Thus $w\in S^3$.

In the general case, let $\roof{x}$ be the roof of a real number
$x$: If $n-1<x\le n$ for some integer $n$, then $\roof{x}=n$.
Define
\[
A=\left\{u_i\wedge u_j\wedge u_k: 1\le \roof{(\tfrac{i}{2})} < \roof{(\tfrac{j}{2})} < \roof{(\tfrac{k}{2})} \le n\right\}.
\]

Clearly $A\subset S^3$ and $|A|=2^3\cdot \binom{n}{3}$.

\bigskip
Step 3. We turn to other kinds of vectors of $S^3$.

Suppose that $w:=u_1\wedge (u_3\wedge u_4-u_5\wedge u_6)$,
$u_2\wedge (u_3\wedge u_4-u_5\wedge u_6)$, $u_3\wedge (u_1\wedge
u_2-u_5\wedge u_6)$, etc., then $\langle \langle u_1\wedge
(u_3\wedge u_4-u_5\wedge u_6)$, $f\wedge u_j^* \rangle \rangle =0$
for all $1\le j\le 2n$. Thus $w\in S^3$. In the general case,
define
\[
B=\{u_i\wedge (u_{2j-1} \wedge u_{2j}-u_{2k-1}\wedge u_{2k}):
\roof{(\tfrac{i}{2})},j,k\mbox{\, are distinct integers}\}.
\]
Then $B\subset S^3$.

Let $W$ be the vector space over $\bm{F}_p$ generated by elements of $B$.
It is not difficult to show that $W$ is also generated by elements of $B_1\cup B_2$ where $B_1$ and $B_2$ are defined by
\begin{align*}
B_1 &= \left\{ u_{2i-1}\wedge \left(\sum_{j\ne i} a_j u_{2j-1}\wedge u_{2j}\right): \sum_{j\ne i} a_j=0\right\}, \\
B_2 &= \left\{u_{2i} \wedge \left(\sum_{j\ne i} b_j u_{2j-1}\wedge u_{2j}\right): \sum_{j\ne i} b_j=0\right\}.
\end{align*}

Elements in $B_1\cup B_2$ are linearly independent over $\bm{F}_p$ and $|B_1\cup B_2|=2n(n-2)$.

It is trivial to check that $\dim_{\bm{F}_p} (\wedge^3 U^*)-2n=2^3\cdot \binom{n}{3}+2n(n-2)$.
We find that elements in $A\cup B_1 \cup B_2$ form a basis of $S^3$.

Obviously, elements in $A\cup B_1\cup B_2$ are of the form $u'\wedge u$ for some $u'\in \wedge^2 U$, $u\in U$.
Thus they belong to $S_{\dec}^3$.
We conclude that $S^3=S_{\dec}^3$.
\end{proof}

\begin{remark}
When $p$ is an odd prime number and $G$ is an extraspecial group
of order $p^{2n+1}$ with $n\ge 2$, we don't know whether
$\bm{C}(G)$ is $\bm{C}$-rational; nor do we know whether
$K_{\max}^3/K^3=H_{\rm nr}^3(\bm{C}(G),\bm{Q}/\bm{Z})$. Similarly,
we don't know the answers to the same questions when $G$ is an
extraspecial group of order $2^{2n+1}$ with $n\ge 3$. As to the
situation when $G$ is an extraspecial group of order $2^5$, it is
known that $\bm{C}(G)$ is $\bm{C}$-rational \cite{CHKP}; thus
$\br_{\rm nr}(\bm{C}(G))=H_{\rm nr}^d
(\bm{C}(G),\bm{Q}/\bm{Z})=\{0\}$ for $d\ge 3$.
\end{remark}

In \cite[page 267]{Pe1}, some algebraic variety $W$ is found such
that ${\rm Br}_{\rm nr}(\bm{C}(W))=\{0 \}$ and $H_{\rm
nr}^4(\bm{C}(W), \bm{Q}/\bm{Z}) \neq \{0\}$, but it is unknown
whether $H_{\rm nr}^3(\bm{C}(W), \bm{Q}/\bm{Z})$ is trivial or
not. The following two propositions provide examples with
analogous phenomena.

%------------------p3.2
\begin{prop} \label{p3.2}
Let $p$ be an odd prime number. Let $G$ be the $p$-group of
exponent $p$ defined by $G=\langle v_i,u_j:1\le i\le 3, 1\le j\le
4\rangle$ satisfying the following conditions

$(1)$ $Z(G)=[G,G]=\langle v_1,v_2,v_3\rangle$, and

$(2)$ $[u_1,u_2]=v_1$, $[u_1,u_3]=[u_2,u_4]=v_2$, $[u_1,u_4]=v_3$,
and the other unlisted commutators are equal to the identity
element of $G$.

Then $\dim_{\bm{F}_p} K_{\max}^2/K^2=1$, ${\rm Br}_{\rm
nr}(\bm{C}(G))\ne \{0\}$, $K^3=K_{\max}^3$, but we don't know
whether $H_{\rm nr}^3(\bm{C}(W), \bm{Q}/\bm{Z})$ is trivial or
not.
\end{prop}

\begin{proof}
Adopt the abbreviation $(2,3)=u_2 \wedge u_3$, etc as before. It
is easy to verify that $S^2=\langle w_1, w_2, w_3 \rangle$ where
$w_1=(2,3), w_2=(3,4), w_3=(1,3)-(2,4)$.

Now we compute $S_{\dec}^2$. We will determine the eligible
elements in $S_{\dec}^2$. Apply Lemma \ref{l2.5} and find the
non-trivial solutions of $(\sum_{1\le i\le 3} a_i\cdot w_i)\wedge
(\sum_{1\le j\le 4} b_j\cdot u_j)=0$. It is not difficult to find
that $S_{\dec}^2 =\langle w_1, w_2 \rangle$. By Theorem
\ref{t2.3}, ${\rm Br}_{\rm nr}(\bm{C}(G))\ne \{0\}$.

On the other hand, it is easy to show that $K^3=\wedge^3 U^*$.
Hence $S^3=\{0\}$. It follows that $S_{\dec}^3=\{0\}$ also.
\end{proof}

%------------------p3.3
\begin{prop} \label{p3.3}
Let $p$ be an odd prime number, $a,b \in \bm{F}_p$ such that the
polynomial $X^2 + aX + b \in \bm{F}_p[X]$ is irreducible. Let $G$
be the $p$-group of exponent $p$ defined by $G=\langle
v_i,u_j:1\le i\le 3, 1\le j\le 4\rangle$ satisfying the following
conditions

$(1)$ $Z(G)=[G,G]=\langle v_1,v_2,v_3\rangle$, and

$(2)$ $[u_1,u_2]=v_1$, $[u_1,u_3]=[u_2,u_4]=v_2$, $[u_2,u_3]=v_3$,
$[u_1,u_4]=v_3^{-b}$, $[u_2,u_4]=v_3^{-a}$, and the other unlisted
commutators are equal to the identity element of $G$.

Then $\dim_{\bm{F}_p} K_{\max}^2/K^2=2$, ${\rm Br}_{\rm
nr}(\bm{C}(G))\ne \{0\}$, $K^3=K_{\max}^3$, but we don't know
whether $H_{\rm nr}^3(\bm{C}(W), \bm{Q}/\bm{Z})$ is trivial or
not.
\end{prop}

\begin{proof}
The proof is similar to the above Proposition \ref{p3.2}. In the
present situation, $S^2=\langle w_1, w_2, w_3\rangle$ where
$w_1=(3,4), w_2=(1,3)-(2,4)-a(2,3), w_3=(1,4)+b(2,3)$, and
$S_{\dec}^2 =\langle w_1 \rangle$. Thus ${\rm Br}_{\rm
nr}(\bm{C}(G))\ne \{0\}$. As in Proposition \ref{p3.2}, we find
that $K^3=\wedge^3 U^*$. Done.
\end{proof}

\bigskip
Now we are going to explain the reason how the $p$-groups of
Theorem \ref{t2.4}, Theorem \ref{t2.6} and Theorem \ref{t2.7} are
found.

Suppose that the group G is of the form $0\to V
\xrightarrow{\iota} G \xrightarrow{\pi} U\to 0$ as in Section 2
and $K^2=\gamma^*(V^*)$ is a subspace of $\wedge^2 U^*$.

Assume that $\dim_{\bm{F}_p}U=6$. We will find a suitable subspace
of $\wedge^2 U^*$ as $K^2$ so that $S_{\dec}^3 \neq S^3$.

For any $w \in \wedge^3 U$, define $X_w := \{ x \in\wedge^2 U^* :
\langle \langle w, x \wedge y \rangle \rangle =0 \, \, for \,
\,any \, \, y \in U^* \}$. It follows that $w \in S^3$ whenever
$K^2$ is a subspace of $X_w$.

We will choose $w$ such that it is highly probable that $w \notin
S_{\dec}^3$ if $K^2$ is chosen judiciously. For this purpose we
choose $w \in \wedge^3 U \setminus ((\wedge^2 U) \wedge U)$. We
will focus on three kinds of such vectors : $u_1 \wedge u_2 \wedge
u_3 +u_3 \wedge u_4 \wedge u_5+u_5 \wedge u_6 \wedge u_1$, $u_1
\wedge u_2 \wedge u_3 +u_3 \wedge u_4 \wedge u_5+u_5 \wedge u_6
\wedge u_1-tu_2 \wedge u_4 \wedge u_6$ (where $t\in \bm{F}_p
\backslash \{0\}$), and $u_1 \wedge u_2 \wedge u_3 +u_4 \wedge u_5
\wedge u_6$. Remember that $u_1, u_2, \ldots, u_6$ is a basis of
$U$.

\medskip
Case 1. $w=u_1 \wedge u_2 \wedge u_3 +u_3 \wedge u_4 \wedge
u_5+u_5 \wedge u_6 \wedge u_1$.

We adopt the convention in Section 2 that $[1,2]$ denotes $u_1^*
\wedge u_2^*$, etc. It is not difficult to find that $X_w= \langle
[1,2]-[4,5], [2,3]-[5,6], [1,4], [2,5], [3,6], [4,6], [3,4]+[1,6],
[2,4], [2,6] \rangle$.

If we choose $K^2$ to be the subspace generated by the first six
vectors of $X_w$, i.e. $K^2=\langle [1,2]-[4,5], \ldots, [4,6]
\rangle$, then we get the same $K^2$ in Peyre's Theorem 3
\cite[page 223]{Pe2}. Thus we get the group of order $p^{12}$
constructed by Peyre.

If we choose $K^2$ to be a 3-dimensional subspace of $X_w$
generated by the elements $[1,2]-[4,5], [2,3]-[5,6]+[1,4],
[3,6]-[2,4]$, we get the group of order $p^9$ in Theorem
\ref{t2.7}.

If we choose $K^2=X_w$, it is not difficult to verify that
$S^3=\langle w, w' \rangle$ where $w'=u_1 \wedge u_3 \wedge u_5$
and $S_{\dec}^3 =\langle w' \rangle$. Thus we find a group of
order $p^{15}$, which is recorded in Theorem \ref{t3.4}.

\medskip
Case 2. $w=u_1 \wedge u_2 \wedge u_3 +u_4 \wedge u_5 \wedge u_6$.

Then $X_w=\langle [1,4], [1,5], [1,6], [2,4], [2,5], [2,6], [3,4],
[3,5], [3,6] \rangle$. If we take $K^2=X_w$, then $S_3= \langle
u_1 \wedge u_2 \wedge u_3, u_4 \wedge u_5 \wedge u_6 \rangle$. It
can be shown that $S_{\dec}^2=S^2$ and $S_{\dec}^3=S^3$. In
conclusion, the process produces no group harmful.

\medskip
Case 3. $w=u_1 \wedge u_2 \wedge u_3 +u_3 \wedge u_4 \wedge
u_5+u_5 \wedge u_6 \wedge u_1-tu_2 \wedge u_4 \wedge u_6$ where
$t\in \bm{F}_p \backslash \{0\}$.

This $w$ is nothing but the $w_1$ in Step 2 of the proof of
Theorem \ref{t2.6}. Thus we may find the group of order $p^9$ in
Theorem \ref{t2.6} if we choose a suitable subspace of $X_w$.

\medskip
Conceivably we may find other ``counter-examples" if choose
various subspaces of $X_w$ for various vectors $w$.

\medskip
The group in the following theorem is found in the above Case 1.

%------------------t3.4
\begin{theorem} \label{t3.4}
Let $p$ be an odd prime number. Let $G$ be the $p$-group of
exponent $p$ defined by $G=\langle v_i,u_j:1\le i\le 9, 1\le j\le
6\rangle$ satisfying the following conditions

$(1)$ $Z(G)=[G,G]=\langle v_1,v_2,v_3\rangle$, and

$(2)$ $[u_1,u_2]=[u_4,u_5]^{-1}=v_1$,
$[u_2,u_3]=[u_5,u_6]^{-1}=v_2$, $[u_1,u_4]=v_3$, $[u_2,u_5]=v_4$,
$[u_3,u_6]=v_5$, $[u_4,u_6]=v_6$, $[u_3,u_4]=[u_1,u_6]^{-1}=v_7$,
$[u_2,u_4]=v_8$, $[u_2,u_6]=v_9$, and the other unlisted
commutators are equal to the identity element of $G$.

Then ${\rm Br}_{\rm nr}(\bm{C}(G))=\{0\}$ and $K_{\max}^3/K^3\ne
\{0\}$.

\end{theorem}

\bigskip

\bigskip
%----------------------------------------References
\renewcommand{\refname}{\centering{References}}

\end{document}